\documentclass{amsart}
\usepackage[utf8]{inputenc}
\usepackage{amsthm}
\usepackage{amsmath}
\usepackage{amssymb}
\usepackage{amscd}
\usepackage{relsize}
\usepackage{tikz-cd}
\usepackage{hyperref}
\hypersetup{
    colorlinks=true,
    linkcolor=cyan,
    filecolor=magenta,      
    urlcolor=blue}
\usepackage{cleveref}
\usepackage[left=2.5cm,top=2.5cm,bottom=3cm,right=2.5cm]{geometry}
\theoremstyle{definition}
\newtheorem{theorem}{Theorem}[section]
\newtheorem{lemma}[theorem]{Lemma}
\newtheorem{proposition}[theorem]{Proposition}
\newtheorem{remark}[theorem]{Remark}

\newtheorem{eg}[theorem]{Example}
\newtheorem{corollary}[theorem]{Corollary}

\newtheorem*{theorem*}{Theorem}
\usepackage{cleveref}

\title{Reducedness of formally unramified algebras over fields}
\author {Alapan Mukhopadhyay and Karen E. Smith}
\thanks{Both authors were partially supported by NSF grant number DMS 1801697}
\date{July 2019}

\begin{document}

\maketitle
\begin{abstract}
We prove that under suitable  graded and local hypothesis, a formally unramified algebra over a field must be reduced. We detail examples, including one due to Gabber, to  show that it is not possible to generalize these results further.
\end{abstract}
\section{Introduction}
Let $A$ be a commutative ring with  multiplicative identity,  and let $S$ be an $A$-algebra. Recall that $S$ is said to be \textit{formally unramified} over $A$ if the module of K\"ahler differentials $\Omega_{S/A}$ is zero. It is well-known that a finitely generated formally unramified algebra over a field $k$ is finite product of separable field extensions of $k$ \cite[Cor 16.16]{Eis}. Since such an algebra is always reduced, this leads to the following natural question:  

\medskip

\noindent
\textbf{Question:}\label{a natural question} Under what hypothesis is a formally unramified algebra over a field reduced?

\medskip

It is easy to find non-reduced formally unramified algebras over an arbitrary field of characteristic $p$ (see \Cref{char p example}).  However, the question is subtle when the ground field has  characteristic zero. For example,   Ofer Gabber proposed a construction of a non-reduced formally unramified algebra over an arbitrary field of characteristic zero, which we explain in \Cref{Gabber}. 

 \medskip
The main purpose of this note is to prove  that, in spite of Gabber's example,   formally unramified extensions of a perfect  field are indeed often reduced. 
For example, we show that if $A$ is a local algebra separated in its $m$-adic topology and  formally unramified over a perfect field,  then $A$ is reduced (\Cref{the local case}). As a corollary, we deduce that any Noetherian $k$-algebra formally unramified over a perfect field is reduced (\Cref{structure of locally Noetherian formally unramified algebra over a perfect field}),
a fact we have not been able to find in the literature though we expect it may be known to experts. We include an example to show that formally unramified  does not imply reducedness for Noetherian local $k$-algebras,  however,  without the assumption that $k$ is perfect  (\Cref{non-example in the Noetherian case}), quite unlike the  finite type case.

\medskip We also get positive results in the graded case: 
\begin{theorem*}
\textit{Let $R$ be an $\mathbb{N}$-graded formally unramified algebra over a perfect field $k$. If the  degree zero graded piece  of $R$  is Noetherian, then  $R$ is reduced}.  
\end{theorem*}

Note that it is {\it not necessary} to assume that the graded ring $R$ is finitely generated over $R_0$, nor that $k\subset R_0$.    We prove this theorem, as well as some variants in which $R_0$ is not assumed to be Noetherian, in Section 3. See Theorem \ref{positive answer} and Remarks \ref{general}, \ref{reduced} and \ref{mixed}.

\bigskip

\noindent \textbf{Notations and conventions:} Every ring in this paper is assumed commutative and with  multiplicative identity. Importantly, we {\it do not} assume that rings are Noetherian, unless we explicitly state so. A triple $(R,m,k)$ represents a (not necessarily Noetherian) local ring $R$, with maximal ideal $m$ and residue field $k$.

\medskip

The set of non-negative integers is denoted by $\mathbb{N}$. The symbol $p$ denotes a positive prime integer.

\medskip

The notation $\Omega_{S/A}$ denotes the module of K\"ahler differentials of an $A$-algebra $S$, and the symbol $d$ denotes the universal derivation 
$S\rightarrow \Omega_{S/A}$. We suppress the dependence on $S$ and $A$ in the notation for $d$ to make the notation less clumsy, so it is important to pay attention to the context (that is, the target module for $d$) when dealing with several algebras or different ground rings. See 
 \cite[Tag00RM]{Stacks} for basics on K\"ahler differentials.

\medskip
\noindent
\textbf{Acknowledgements:} The authors thank Mel Hochster for kindly sharing his notes on Gabber's example (\Cref{Gabber}). We are grateful  to Shubhodip Mondal for bringing our attention to these issues, while relaying  related questions that were posed in the Arizona Winter School 2019 (see Section 4). 

\section{Examples of non-reduced formally unramified algebras}

In this section, we discuss examples of non-reduced formally unramified algebras. 
We first recall that it is easy to find such examples in prime characteristic.
\begin{eg}\label{char p example}
Fix a field  $k$ of positive characteristic $p.$ Let  $k[X]$ be the polynomial ring in one variable over $k$, and denote by  $L$ an algebraic closure of its fraction field. For $n\in \mathbb{N}$, let $X^{1/p^n} \in L$ be the $p^n$-th root of $X$. 
Consider the $k$-subalgebra $k[X^{1/p^{\infty}}]$
of $L$ generated by the $X^{1/p^n}$ as we range over all natural numbers $n$.
Let $A$ be the quotient ring 
$\frac{k[X^{1/p^{\infty}}]}{(X)}$. 

Denote the image of $X$ in $A$ by $x$. Observe that  $d(x^{1/p^n})= d((x^{1/p^{n+1}})^p)=0$ for each $n \in \mathbb{N}$. Since the
$d(x^{1/p^n})$ generate $\Omega_{A/k}$ as an $A$-module, we see that $\Omega_{A/k}$ is zero. 
But, of course, the algebra $A$ is non-reduced: for example $x^{1/p}$ in $A$ is one of many non-zero nilpotent elements.
\end{eg}

We now describe an example of Ofer Gabber:

\begin{theorem}\label{Gabber}
Fix any field $k$ of characteristic zero. There exists a formally unramified local $k$-algebra that is not reduced.
\end{theorem}

\smallskip
\noindent
\begin{remark}
 Gabber's example in Theorem \ref{Gabber} is necessarily non-Noetherian, as is Example \ref{char p example}. As we soon prove, Noetherian local formally unramified algebras over a perfect field are always reduced; see Corollary \ref{structure of locally Noetherian formally unramified algebra over a perfect
field}.
\end{remark}

\medskip
\begin{proof}[Proof of Theorem \ref{Gabber}]
 We will construct  a direct limit of local  $k$-algebras
\begin{equation}\label{sequence}
(R_0, m_0, k) \hookrightarrow (R_1, m_1, k) \hookrightarrow (R_2, m_2, k) \hookrightarrow 
\dots
\end{equation}
satisfying
\renewcommand{\theenumi}{\roman{enumi}}
\begin{enumerate}
    \item $k\hookrightarrow R_0$ is  a proper inclusion.
    \item Each algebra $R_i$ is a finite dimensional local $k$-algebra,  with the natural composition $k\hookrightarrow R \twoheadrightarrow R_i/m_i $ an isomorphism.
    
    \item The local $k$-algebra inclusion $R_i \subseteq R_{i+1}$ induces the zero map $\Omega_{R_i/k}\rightarrow \Omega_{R_{i+1}/k}$ for each $i \in \mathbb{N}$.
\end{enumerate}
Then to produce the local $k$-algebra satisfying the conclusion of Proposition \ref{Gabber}, we take the direct limit, setting
$$R_{\infty}:= \underset{i \in \mathbb{N}}{\varinjlim}\, R_i, \, \, \, \, m_{\infty}:= \underset{i \in \mathbb{N}}{\varinjlim} \, m_i.$$
The $k$-algebra $R_{\infty}$ is a local ring with a non-zero  maximal ideal $m_{\infty}$. Since the construction of modules of Kahler differentials commutes with taking direct limits 
\cite[Tag 00RM]{Stacks}),  we have $\Omega_{R_{\infty}/k}= \underset{i \in \mathbb{N}} {\varinjlim}\, \Omega_{R_i/k}=0$. Furthermore, because each $R_i$ is finite dimensional, we know  each $m_i$ is nilpotent, and so each element of $m_{\infty}$ is nilpotent. So $(R_{\infty}, m_{\infty}, k)$ serves as an example proving Theorem \ref{Gabber}.

\medskip
 
 To construct the sequence (\ref{sequence}), we begin by taking $R_0$ to be $B$ as defined in the next lemma.

\begin{lemma}\label{preparatory}
Fix a field $k$ of characteristic zero, and an integer $n\geq 5$. 
Let $F = X^2Y^2+X^n+Y^n$ in the power series ring $k[[X,Y]]$. Let  $B$ denote the quotient ring $k[[X, Y]]/(\frac{\partial F}{\partial X}, \frac{\partial F}{\partial Y})$.
 Then,
\renewcommand{\theenumi}{\roman{enumi}}
\begin{enumerate}
 \item $B$ is an Artinian local ring; that is, $\text{dim}_k({B})< \infty$ 
 \item The image $f$ of $F$ in the quotient ring $B$ is not zero but its square is zero.
 \item The element $df$ is zero in $\Omega_{B/k}$.
\end{enumerate}
\end{lemma}

Before proving Lemma \ref{preparatory}, we  point out how to use it to  construct the sequence (\ref{sequence}) of algebras $R_i$.   Having set  $R_0 = B$, we  can then inductively produce $(R_i, m_i, k)$ from $(R_{i-1}, m_{i-1}, k)$  using  Lemma \ref{killing differentials} below.  

\begin{lemma}\label{killing differentials}
Fix a field $k$ of characteristic zero. Let $(R, m, k)$ be any local $k$-algebra of finite dimension over $k$ such that  the composite map $k \hookrightarrow R \rightarrow R/m$ is an isomorphism. Then there exists a finite dimensional local $k$-algebra extension  $(R, m, k)\hookrightarrow (\widetilde{R},  \widetilde{m}, k)$ (with
$k \hookrightarrow \tilde{R} \rightarrow \tilde{R}/\tilde{m}$ also an isomorphism) such that 
the induced map $\Omega_{R/k} \rightarrow \Omega_{\widetilde {R}/k}$ is the zero map.

\end{lemma}

Thus Theorem \ref{Gabber} is proved as soon as we prove the preceding two
lemmas,   which we now do in turn, using Lemma \ref{preparatory} to establish Lemma \ref{killing differentials}.

\begin{proof}[Proof of Lemma \ref{preparatory}]
For (i), we use the fact that  $k[[X, Y]]$ is a two-dimensional UFD \cite{Aus}. Since $\frac{\partial F}{\partial X},  \frac{\partial F}{\partial Y}$ have no common factors,
they form a regular sequence in the power series ring $k[[X, Y]]$. In particular, the ideal $(\frac{\partial F}{\partial X},  \frac{\partial F}{\partial Y})$ is primary to the maximal ideal, so that the quotient $B = k[[X, Y]]/(\frac{\partial F}{\partial X},  \frac{\partial F}{\partial Y}) $ is finite dimensional over $k$.

For (ii), we introduce some notation. Set
 $F_1= 2Y^2+nX^{n-2}$ and  $F_2 = 2X^2+ nY^{n-2}$, so that  $\frac{\partial F}{\partial X}= XF_1$ and $ \frac{\partial F}{\partial Y}= YF_2$.  We also use lower case letters to indicate images in the quotient $B =  k[[X, Y]]/(\frac{\partial F}{\partial X},  \frac{\partial F}{\partial Y})$. 
 
 Note that 
\begin{equation}\label{feq}
F-\frac{1}{n}(X\frac{\partial F}{\partial X}+ Y\frac{\partial F}{\partial Y})=(1- \frac{4}{n})X^2Y^2,
\end{equation}
so that $f = (1-\frac{4}{n})(xy)^2$ in $B$. To see that this element is non-zero, we argue by contradiction. Lifting to the power series ring, 
the statement that $f=0$ would mean that  
$$X^2Y^2 = G\frac{\partial F}{\partial X}+ H\frac{\partial F}{\partial Y} = GXF_1 + HYF_2$$
 for some power series $G, H$.
 Using the unique factorization property, we see that $G = YG_1$ and $H=XH_1$ for some power series $G_1, H_1$.  So, we get $XY= G_1F_1+ H_1F_2$. Comparing order two terms on both sides of the last equality, we get a contradiction. 
 So $f\neq 0$ in $B$.

To show $f^2=0$ in $B$, we again invoke Equation (\ref{feq}), observing that it is enough to show that $(x^2y^2)^4 = 0$ in $B$. In fact, we'll show that $xy^3 = 0$ in $B$. Computing in the power series ring, we have
$$
F_1 - \frac{n}{2}X^{n-4}(F_2) = Y^2(2-\frac{n^2}{2}X^{n-4}Y^{n-2}) \in (F_1, F_2)
$$
and since $2-\frac{n^2}{2}X^{n-4}Y^{n-2}$ is a unit in $k[[X, Y]]$, we have  $Y^2 \in (F_1, F_2).$ Multiplying by $XY$, we conclude that 
$$
XY^3 \in  (XF_1, YF_2) = (\frac{\partial F}{\partial X},  \frac{\partial F}{\partial Y}),
$$
whence $xy^3 = 0$ in $B$. In particular, $f^2 = x^4y^4 = 0$ in $B$.

(iii)  The universal derivation $B\rightarrow \Omega_{B/k}$ sends $f$ to $df = 
\frac{\partial F}{\partial X} dx + \frac{\partial F}{\partial Y}dy = 0$, where abusing notation, the notation 
$\frac{\partial F}{\partial X} $ and $\frac{\partial F}{\partial Y} $ denotes the images in $B$. 
Since these coefficient are zero in $B$, we see that $df = 0$. 

\end{proof}

\begin{proof}[Proof of \Cref{killing differentials}]
Fix a finite dimensional algebra $(R, m, k)$ as in the lemma. It suffices to show that, for a given non-zero
 $r \in m$, we can construct a finite dimensional  local $k$-algebra $(R', m', k)$
 (with composite $k \hookrightarrow R' \twoheadrightarrow R'/m'$ an isomorphism)
 and a local $k$-algebra injection $R \hookrightarrow R'$ such that the induced map $\Omega_{R/k} \rightarrow \Omega_{R'/k}$ sends $dr$ to zero.
 Indeed, first observe that our hypothesis implies that $\Omega_{R/k}$ is generated by elements $r$ where $r\in m$. So choosing a $k$-basis $e_1, e_2, \ldots, e_l$ for $m$,
 it is clear that we simply need to repeat the construction $l$   times (first for $r=e_1$ and then for $r$ the image of $e_2$ and so on), to get $(\tilde{R}, \tilde{m}, k)$ together with a local $k$-algebra injection $R \hookrightarrow \tilde{R}$ such that the induced  map $\Omega_{R/k} \rightarrow \Omega_{\tilde{R}/k}$ sends each $de_i$ to zero. So $\tilde R$ can be taken to be the finite dimensional local $k$-algebra guaranteed by the lemma.
 
 \medskip
 So fix non-zero $r \in m $. Let $t \in \mathbb{N}$ be such that $r^t=0$ but $r^{t-1}\neq 0$. Now, for $B$ as in \Cref{preparatory}, define $B_{t}$ to be the tensor product of $t-1$ copies of $B$ over $k$. For $1 \leq i \leq t-1$, let $g_i= 1 \otimes \ldots \otimes f \otimes \ldots 1 \in B_t$ (where $f$ is at the $i$-th spot), and set $g= \sum \limits_{i=1}^{t-1} g_i$.  Define $R'= R \otimes_k B_t/(r\otimes 1- 1 \otimes g)$.
 
Observe that there is a canonical $k$-algebra  map $\iota: R \rightarrow R'$, and that  $R'$ is a finite dimensional local $k$-algebra with residue field $k$. It remains to show that the induced map $\Omega_{R/k} \rightarrow \Omega_{R'/k}$ sends $dr$ to zero and that $\iota$ is injective. To the first end, we invoke 
\Cref{preparatory}, which tells us  that that $df =0$ in $\Omega_{B/k}$, so that each $dg_i$ is zero in 
$\Omega_{B_t/k}$, thus $dg$ is zero in $\Omega_{B_t/k}$ as well. Because $g$ and $r$ get identified in $R'$, it follows also that the 
natural image of $dr$ in $\Omega_{R'/k}$ is $dg =0$ as well. That is, the image of $dr $ under the natural map  
$ \Omega_{R/k} \rightarrow \Omega_{R'/k}
$
induced by $\iota$ is zero, as needed. 

 \medskip
It remains only to show that  $\iota: R \rightarrow R'$ is injective. For this, observe that in $B_t$, we have that $g_i^2= 0$ for all $i$, so that $g^t=0$; however, $g^{t-1}= (t-1)!f \otimes f \otimes \ldots \otimes f$ is not zero. 
Set $A_t = k[z]/z^t$. 
We have $k$-algebra injections $A_t \hookrightarrow R$ and $A_t \hookrightarrow B_t$ sending $z$ to $r$ and  to $g$ respectively, 
so that we can consider both  $R$ and $B_t$ as $A_t$-algebras. The natural map $R \otimes_k B_t \rightarrow R \otimes_{A_t}B_t$ kills $r\otimes1 - 1 \otimes g$, so it factors through $R'$.  Thus to prove the injectivity of $\iota$,  it suffices to show that the composite $R \overset{\iota}\longrightarrow R' \rightarrow R \otimes_{A_t}B_t$ is injective. For this, note that  $A_t$ is finite dimensional Gorenstein algebra, and hence injective as an $A_t$ module \cite[Prop 21.5]{Eis}, which means that the $A_t$-module map $A_t \hookrightarrow B_t$ splits. Tensoring with $R$, we see that the composition $R \rightarrow R \otimes_{A_t}B_t$, and hence $\iota,$
splits as well. So $\iota$ is injective.

\end{proof}

\end{proof}

\section{Reducedness of Local and Graded Unramified algebras}
In this section, we establish affirmative answers to Question 1 about the reduced-ness of formally unramified algebras.
We first point out a straightforward result in the local case under suitable finiteness conditions:
\begin{theorem}\label{the local case}
Let $(R,m,k)$ be a local algebra over a field $L$.  Assume that 
\begin{enumerate}
    \item $R$ is $m$-adically separated, meaning that $\underset{n \in \mathbb{N}}{\bigcap}m^n=0$; and
    \item The field extension given by the composite $L \rightarrow R \rightarrow k$ is {\it separable.}{\footnote{Recall that  {\it algebraic} field extension $L \subseteq k$ is  \textbf{separable} if the minimal polynomial of any element of $k$ over $L$ has distinct roots in $\overline L$. 
An arbitrary field extension $L\subseteq k$ is \textbf{separable}  if for every sub-extension $L \subseteq L' \subseteq k$ with $L'$ finitely generated over $L$, $L'$ admits a transcendence basis $\{X_1, \ldots, X_r\}$ over $L$, such that $L(X_1, \ldots, X_r) \subseteq L'$ is a separable algebraic extension.}}
\end{enumerate}
\smallskip
  If $\Omega_{R/L}=0$, then $R$ is a field. 
\end{theorem}

\medskip
The   separability assumption in  \Cref{the local case} is necessary, as the following example shows:
\begin{eg}\label{non-example in the Noetherian case}
Fix $L= \mathbb{F}_p(x)$, the function field in the variable $x$ over $\mathbb{F}_p$
and let $k= \mathbb{F}_p(x^{\frac{1}{p^{\infty}}})$ be the perfection of $L$. We will construct a Noetherian local $k$-algebra which is formally unramified over $L$ but which is not reduced.

For $f(x) \in L$, let $f'(x)$ denote  the derivative of $f(x)$ with respect to $x$.  Viewing $L$ as a subfield of $k$, we can also view $f'(x)$ as an element in $k$.
Set $A= k[Z]/(Z^2)$, and let $z$ be the image of $Z \in k[Z]$ in $A$. Consider the
additive map $\phi: L \rightarrow A$ given by $\phi(f(x))= f(x)+f'(x)z$. It is not hard to verify that $\phi$ is a ring homomorphism, using the fact that $z^2=0$. View $A$ as an $L$-algebra using this map (note: we are {\it not} using the ``obvious" $L$ structure induced by the inclusions $L\subset k \subset A$). 
Clearly $A$ is a non-reduced Noetherian local $L$-algebra. 

We now verify that $\Omega_{A/L}=0$. For this, it suffices to check that the differential $da$ is zero  in $\Omega_{A/L}$  for each  $a \in A.$ 
Since $k$ is perfect, we can write  $a= g_1^p + g_2^pz$, for some $g_1, g_2 \in k$. So $da=g_2^pdz$. Now, $dz$ is zero since $dz= d((x^{1/p})^p+z)= d(\phi(x))$.
\end{eg}



\bigskip

Before proving \Cref{the local case}, 
we point out some  consequences. 
\begin{corollary}\label{structure of locally Noetherian formally unramified algebra over a perfect
field} Fix a perfect{\footnote{Recall that the  field  $k$ is \textbf{perfect} if every field extension is separable. Equivalently, 
 $k$ is  perfect if and only if either $k$ has characteristic zero or $k=k^p$, where $p>0$ is the characteristic of $k$; see \cite[Tag 05DU]{Stacks}.
}}
  ground field $k$.
\renewcommand{\theenumi}{\roman{enumi}}
\begin{enumerate}
    \item Let $A$ be a $k$-algebra with the property that the localization  at every maximal ideal is Noetherian. If $\Omega_{A/k}=0$, then $A$ is reduced.
    \item Let $A$ be a Noetherian $k$-algebra. If $\Omega_{A/k}=0$, then $A$ is finite product of perfect fields.
\end{enumerate}
\end{corollary}

\begin{proof}[Proof of Corollary \ref{structure of locally Noetherian formally unramified algebra over a perfect
field}]
For (i),  observe that it  is enough to show that the localization  $A_m$ is reduced where $m$ is an arbitrary maximal ideal of $A$. Since the formation of module of K\"ahler differentials commutes with localization  \cite[Prop 16.9]{Eis}, we have $\Omega_{A_m/k}=0$. By Krull's intersection theorem, we know  $\underset{n \in \mathbb{N}}{\bigcap}m^nA_m=0$, whence   \Cref{the local case} implies that  $A_m$ is reduced.

For (ii), note first that Theorem \ref{the local case} implies immediately that the local ring of $A$ at a maximal ideal is a field. So every maximal ideal of $A$ is minimal. Since $A$ is Noetherian, $A$ has only finitely many minimal--and hence maximal-- ideals,  say $m_1, \ldots, m_r$. The Chinese Remainder Theorem tells us that
the canonical map
$$A \rightarrow \prod \limits_{i=1}^r A/{m_i}$$
 is surjective with  kernel equal to  the nilradical of $A$. 
 By part (i), $A$ is reduced, and hence this  map is an isomorphism, and   $A$ is a product of fields. 
 It remains only to show that each $A/m_i$ is perfect.  When $k$ has characteristic zero this is immediate. When $k$ has positive characteristic, observe that $\Omega_{A/k} = \bigoplus_{i=1}^r \Omega_{(A/m_i)/k}$, so that each  $\Omega_{(\frac{A}{m_i})/k}=0$.
 But now we can invoke the following lemma
 to complete the argument: 
 {\it{
Let $k$ be a perfect field of characteristic $p>0$, and suppose $k \subseteq K$ is a field extension. For $x \in K$, $dx=0 \in \Omega_{K/k}$ if and only if $x$ has a $p$-th root in $K$}}  \cite[Tag 031U]{Stacks}.
\end{proof}

 \Cref{the local case} follows from the following special case:
\begin{lemma}\label{some power of the maximal ideal is zero}
Let $(R,m,k)$ be a local algebra over a field $L$ such that the field extension given by the composite $L \rightarrow R \rightarrow k$ is (possibly non-algebraic) separable. If some power of $m$ is zero and $\Omega_{R/L}=0$, then $R$ is a field.
\end{lemma}
\begin{proof}[Proof of Lemma] This is slightly subtle since we can not assume  the extension $L\hookrightarrow k$ is algebraic.  Because, some power of $m$ is zero, the local  ring $(R,m)$ is complete. Following the proof of Cohen structure theorem \cite[Tag 0323]{Stacks}, we can find an $L$-algebra map $k \rightarrow R$ such that the composite $k \rightarrow R \rightarrow R/m$ is an isomorphism. In this case, because $(R, m)$ contains a copy of its residue field $k$, we have an 
isomorphism (see  \cite[Tag 0B2E]{Stacks})
\begin{equation}\label{iso}
m/m^2 \rightarrow  R/m \otimes_R \Omega_{R/k}.
\end{equation}
 Now, since $L\subseteq k \subseteq R$, 
our assumption that $\Omega_{R/L}=0$ implies also   $\Omega_{R/k}=0$.
So the isomorphim (\ref{iso}) implies 
that $m=m^2$, and hence $m=m^n$ for all natural numbers $n$. Combined with the assumption that some power of $m$ is zero, we conclude that $m=0$. That is,  $R$ is a field.
\end{proof}

\begin{proof}[Proof of Theorem \ref{the local case}]
Let $R' = R/m^2$ and note that $R'$ satisfies the hypothesis of Lemma \ref{some power of the maximal ideal is zero}. Since $R\rightarrow R'$ is surjective, also 
$\Omega_{R/L} \rightarrow \Omega_{R'/L}$ is surjective.  So our hypothesis that $\Omega_{R/L}  = 0$   implies that also $\Omega_{R'/L}=0$.  Now using  \Cref{some power of the maximal ideal is zero}, we see that the maximal ideal of $R/m^2$ is zero. So $m=m^2$ in $R$, from which it follows that  $\bigcap_{n\in\mathbb{N}}m^n = m$ in $R$. But now our hypothesis that $R$ is $m$-adically separated implies that 
 $m=0$, completing the proof that  $R$ is a field.
\end{proof}


\begin{remark}
If an essentially finite type algebra over a field is formally unramified, then the algebra is integral over the field  \cite[Tag 02G3]{Stacks}. One might wonder whether the same is true with out the essentially finite type hypothesis. Although in positive characteristic this need not be the case, in characteristic zero this turns out to be true (see \cite{AS}).
\end{remark}

\vspace{5mm}
\subsection{The graded case.}

Fix a field $k$. By ``$\mathbb N$-graded $k$-algebra"  we mean a $k$-algebra $R$ whose underlying additive group 
admits a decomposition  $\bigoplus_{n\in \mathbb N} R_n$ with the property that $R_n\cdot R_m \subset R_{m+n}$ for all $m, n\in \mathbb N$.  Specifically, we do not make the common assumption  $k\subset R_0$.

\begin{theorem}\label{positive answer} Fix a perfect ground field $k$.
Let $R$ be an $\mathbb{N}$-graded $k$-algebra for which the subring $R_0$ is Noetherian. If $R$ is formally unramified over $k$, then $R=R_0$ and $R$ is reduced.
\end{theorem}

\begin{remark}\label{general}
 Alternatively, as the proof will  show, rather than assuming that $k$ is perfect and $R_0$ is Noetherian in Theorem \ref{positive answer}, we may instead  assume $(R_0, m)$ is local and $m$-adically separated with residue field a separable extension of
 $k$ or other variants which imply that $R_0$ satisfies the hypotheses of  Theorem \ref{the local case}.
 \end{remark}
 
\begin{proof}
 For $t \in \mathbb{N}$, set $R_{\geq t}$ to be the ideal generated by homogeneous elements of degree $t$ or higher. Fix $t \geq 1$ and set $R'= R/R_{\geq t}$. Note that $R'$ is $\mathbb{N}$-graded and formally unramified over $k$,  where the $k$-algebra structure on $R'$ comes via the composition $k \rightarrow R \rightarrow R/R_{\geq t}=R'$. In particular, since $k$ is perfect and $R_0$ is Noetherian, we can apply Corollary \ref{structure of locally Noetherian formally unramified algebra over a perfect
field} to the $t=1$ case to conclude that $R_0$ is a finite product of fields.

\medskip
 We now claim that the inclusion $R_0 \rightarrow R/R_{\geq t} = R'$ is an isomorphism. Once we have this for every $t \geq 1$, it is clear that $R=R_0$ and that $R_0$ is reduced.
 
 \medskip
 To this end, fix $t$ and consider the graded $k$-algebra $R' = R/R_{\geq{t}}.$  To show that $R_0 \rightarrow R'$ is an isomorphism, it suffices to show that for every maximal ideal $m$ of $R'$, the ideal $I  = R_1\oplus R_2\oplus \cdots \oplus R_{t-1}$ of $R'$ becomes zero after localization at $m$. 
 
 Fix an arbitrary maximal ideal $m$ of $R'$. Since the elements of $I$ are all nilpotent, $I$ is contained in every prime of $R'$ including $m$.
 Of course, $m_0 = R_0\cap m$ is also contained in $m$, and since $R_0$ is product of fields we know that $m_0$ is maximal in $R_0$. 
 It follows that our arbitrary ideal $m$ has the form   $ m_0 \oplus I = m_0\oplus R_1\cdots \oplus R_{t-1}.$

 Now to show that the localization $I_m$ is zero, it suffices to show that that the localization 
 $I_{m_0} = I \otimes_{R_0} (R_{0})_{m_0}$ is zero,  since $I_m$ can be obtained from $I_{m_0}$ by further localization at the multiplicatively closed set 
  $R\setminus m \supset R_0\setminus m_0$.  For this,  tensor over $R_0$ with the field  $(R_0)_{m_0} = L_0$ to produce 
 $$
(R_0)_{m_0} \oplus (R_1)_{m_0}\oplus \ldots \oplus(R_{t-1})_{m_0}
  $$
  which we denote by $R'_{m_0}$.  Note that $R'_{m_0}$ is a 
   local algebra over the perfect field $k $ with maximal ideal $ (R_1)_{m_0}\oplus \ldots \oplus(R_{t-1})_{m_0}= I_{m_0}. $
   Furthermore, being a localization of the formally unramified  $k$-algebra $R'$, we know also that $R'_{m_0}$ is formally unramified.
   Finally, because the $t$-th power of the maximal ideal $I_{m_0}$ is zero, we see that $R'_{m_0}$ satisfies the separation hypothesis of 
   Theorem \ref{the local case}. So invoking that theorem, $R_{m_0}'$ is a field, and its maximal ideal $I_{m_0}$ is zero. This completes the proof.

\end{proof}

\medskip

We wish to give another proof of Theorem \ref{positive answer} in the case where $k\subset R_0$ using the following result about the kernel of the universal derivation for a graded ring.

\begin{proposition}\label{kernel of the universal derivation}
Let $R$ be an $\mathbb{N}$-graded ring containing a field $k$. 
\renewcommand{\theenumi}{\roman{enumi}}
\begin{enumerate}
    \item If $k$ has characteristic zero, then the kernel of the universal derivation
    $d: R \rightarrow \Omega_{R/R_0}$ is $R_0$.
    \item If  $k$ has characteristic $p>0$, then  the kernel of the universal derivation $d:R \rightarrow \Omega_{R/R_0}$ is contained in the $p$-th Veronese subring of $R$:
    $$
\text{ker}(d: R \rightarrow \Omega_{R/R_0}) \,\,\subseteq \,\,\underset{j \in \mathbb{N}}{\bigoplus}R_{jp}.$$

\end{enumerate}
\end{proposition}

\begin{proof}[Proof that Proposition \ref{kernel of the universal derivation} implies Theorem \ref{positive answer} when $k\subset R_0$]
Because $\Omega_{R/k} = 0$, we immediately know that  $\Omega_{R/R_0} = 0$ as well. 
 We claim that this implies that $R=R_0$. 
 Indeed, in this case,
 the universal derivation 
 $$R \overset{d}\longrightarrow \Omega_{R/R_0}$$
is zero. So using  \Cref{kernel of the universal derivation}, we see that 
 \begin{enumerate}
     \item In characteristic zero, $R_0 = R$, establishing the claim immediately; whereas
     \item In characteristic $p>0$,  any non-zero homogeneous element of $R$
     has degree a multiple of $p$.       But then we can re-grade $R$,
     setting its degree $j$-th piece to be $R_{jp}$. With this new grading, Theorem \ref{kernel of the universal derivation} again implies that the non-zero homogeneous elements of $R$ have degree a multiple of $p$, which of course means their degrees are multiples of  $p^2$ using the original grading. Again regrading and  iterating this procedure, we see that any homogeneous element of $R$ must have degree a multiple of $p^e$ for all $e$. This forces $R=R_0$.
 \end{enumerate}

 Finally, since $R=R_0$ is Noetherian and $k$ is perfect, Theorem \ref{positive answer} follows immediately from Corollary \ref{structure of locally Noetherian formally unramified algebra over a perfect field}.  
 \end{proof}

 \medskip
 \begin{remark} Even if $k$ is not assumed to be contained in $R_0$, Proposition \ref{kernel of the universal derivation}  can be adapted to prove Theorem \ref{positive answer} in characteristic $p>0$ or in characteristic zero if we assume $k$ is algebraic over its prime field.
 
 Indeed, because the multiplicative identity of $R$ must have degree zero, the prime field $\mathbb F$ of $k$ (which is $\mathbb F_p$ in characteristic $p>0$ or $\mathbb Q$ in characteristic zero) is contained in  $R_0$. 
 Our assumptions on $k$ imply that $\Omega_{k/\mathbb F}$ is zero \cite[Thm 25.3]{Matsumura} whereas our hypothesis that $R$ is formally unramfied over $k$ ensures that $\Omega_{R/k}$ is zero. 
 So the 
 exact sequence \cite[Prop 16.2]{Eis}
 $$ 
 R \otimes_k \Omega_{k/\mathbb F} \rightarrow \Omega_{R/\mathbb F} \rightarrow \Omega_{R/k} \rightarrow 0
$$
 guarantees that also $ \Omega_{R/\mathbb F} $ is zero. Since $\mathbb F \subset R_0$,  arguing as in (1) and (2), we can conclude that $R$ is reduced and concentrated in degree zero from \Cref{kernel of the universal derivation}.
 \end{remark}
  
 \begin{remark}\label{reduced} 
 We can use the line of argument sketched in points (1) and (2) above to prove the reducedness of a graded algebra, without any Noetherian assumptions, in the following context. \textit{Suppose that $R$ is an $\mathbb{N}$-graded algebra.  If $R$ is formally unramified over any field $k$ contained in $R_0$, then $R=R_0$. Thus if we further assume $R_0$ is reduced, we get $R$ is also reduced. } See also Remark \ref{mixed} for a statement when $R$ does not contain a field.
 \end{remark}
  
\medskip
It remains only to  prove \Cref{kernel of the universal derivation}.
For this, we make  use of the {\bf Euler operator}, suitably interpreted for a (possibly) infinitely generated polynomial ring over an arbitrary ground ring. The next lemma tells us that the Euler operator behaves especially well on homogeneous elements.

\begin{lemma}\textbf{[Euler's Homogeneous Function Theorem]}\label{Euler operator}
Let $P = A[\{X_{\alpha}\}]$ be a polynomial ring over an arbitrary ground ring $A$ in (possibly infinitely many) variables $X_{\alpha}$ indexed by the set $\Lambda$. Assume that $P$ is $\mathbb N$-graded, with  each $X_{\alpha}$ homogeneous of non-zero degree and elements of the coefficient ring $A$ has degree zero. Then for any homogeneous $G \in P$,
\begin{equation}\label{Euler}
\underset{\alpha \in \Lambda}{\sum}\text{deg}(X_{\alpha})\cdot X_{\alpha}\cdot\frac{\partial{G}}{\partial{X_{\alpha}}}= \text{deg}(G)\cdot G.
\end{equation}
\end{lemma}
\begin{proof} First note that since $G$ involves only finitely many $X_{\alpha}$, the sum  on the left side is finite. To verify Equation (\ref{Euler}), it suffices to  check  the case when $G$ is a monomial in $X_{\alpha}$, where it  follows  from a direct calculation.
\end{proof}

\medskip
\begin{proof}[Proof of \Cref{kernel of the universal derivation}]

Fix an $\mathbb N$-graded algebra $R$.  Note that, considered as a graded algebra over the subring $R_0$ of degree zero elements,  $R$ can be generated by homogeneous elements, say $\{r_{\alpha}\}_{\alpha\in \Lambda},$ of positive degree (where $\Lambda$ is some arbitrary indexing set).  We thus have a graded $R_0$-algebra presentation for $R$
$$
P := R_0[\{X_{\alpha}\}]_{\alpha\in \Lambda} \,\, \overset{\pi}\twoheadrightarrow \,\, R
$$
sending each variable $X_{\alpha}$ of the polynomial ring $P$ to the correspondingly-indexed element $r_{\alpha}$ in our generating set. This presentation preserves degree provided we grade $P$ so that $X_{\alpha}$ is assigned degree equal to the degree of $r_{\alpha}$. Of course $\pi$ induces an isomorphism  $P/I\cong R$, so we can identify $R$ with $P/I$ as graded rings.  For the remainder of the argument, we adopt the convention that upper case letters denote elements of $P$ and lower case letters are their corresponding images in $P/I=R$.

Let us now examine the universal derivation $d: R\rightarrow \Omega_{R/R_0}$. Since $\Omega_{R/R_0}$ carries an induced grading which makes $d$ degree preserving, the kernel of $d$ is generated by homogeneous
elements. Let $f\in R = P/I$ be a homogeneous element in $\ker d$ and let $F$ be a homogeneous lift to the polynomial ring $P$.

Because $P$ is a polynomial ring over $R_0$, the module of K\"ahler differentials $\Omega_{P/R_0}$ is a free $P$-module on the generators $dX_{\alpha}$ (where $\alpha$ ranges through  $\Lambda$).  Using the
the conormal exact sequence \cite[Prop 16.3]{Eis}, 
$$
 R \otimes_P I  \xrightarrow{1\otimes d} R \otimes_P \Omega_{P/R_0} \rightarrow \Omega_{R/R_0} \rightarrow 0,
$$
we see that  $df = 0$  in 
$ \Omega_{R/R_0}$ means that 
we can find  homogeneous $G_1, \ldots, G_m \in I\subset P$ and  $h_1, \ldots, h_m \in R$ and such that
$$1 \otimes dF \, =   \sum \limits_{i=1}^m h_i\otimes dG_i \, \in \,  R \otimes_P \Omega_{P/R_0}\,  \cong \, \bigoplus_{\alpha\in \Lambda} R \, dX_{\alpha}.
$$
Explicitly, we unravel this by computing in the free module $R\otimes \Omega_{P/R_0}$ that
\begin{equation}\label{eq1}
\underset{\alpha \in \Lambda}{\sum}\frac{\partial{F}}{\partial {X_{\alpha}}}dX_{\alpha} \, = \, \sum \limits_{i=1}^{m}h_i(\underset{\alpha \in \Lambda}{\sum}\frac{\partial{G_i}}{\partial{X_{\alpha}}}dX_{\alpha}) \, \, \in \, \underset{\alpha \in \Lambda}{\bigoplus} R\, dX_{\alpha}, 
\end{equation}
where the coefficients $\frac{\partial{F}}{\partial {X_{\alpha}}}$ and
$\frac{\partial{G_i}}{\partial {X_{\alpha}}}$ should be understood as their natural images  in $P/I = R$. 
Comparing the coefficients of the free generators $dX_{\alpha}$ on either side of Equation (\ref{eq1}), we see that 
$$\frac{\partial{F}}{\partial{X_{\alpha}}}- \sum \limits_{i=1}^m H_i \frac{\partial{G_i}}{\partial{X_{\alpha}}}  \in I
$$
where each $H_i$ is a lift of $h_i$ to $P$.
Multiplying by $X_{\alpha}$ and summing up to get the Euler operator, we 
have
$$
 \underset{\alpha \in \Lambda}{\sum} \text{deg}(X_{\alpha})\cdot X_{\alpha}\cdot \frac{\partial{F}}{\partial{X_{\alpha}}}  \,\,\, - \,\,\, \sum \limits_{i=1}^m H_i(\underset{\alpha \in \Lambda}{\sum} \text{deg}(X_{\alpha})\cdot X_{\alpha}\cdot \frac{\partial{G_i}}{\partial{X_{\alpha}}}) \,\, \in \, \, I 
 $$
 so that in light of \Cref{Euler operator}, we have that 
 $$
\text{deg}(F)\cdot F  -\sum \limits_{i=1}^m H_i\cdot (\text{deg}(G_i))\cdot G_i \,\, \in \,\, I. 
$$
In particular, since the $G_i\in I$,  we see that  $\text{deg}(F) F  \in I,$ and we can conclude that  
\begin{equation}\label{eq3}
\text{deg}(f) f  = 0 
\end{equation}
in $R$.

The proof of the proposition now follows easily. 
In  case (i),  $\deg{f}$ is a unit in $R$  if it is non-zero. So no positive degree $f$ can be in the kernel of $d:R\rightarrow \Omega_{R/R_0}$. That is,   $\text{ker}(d)= R_0$.
In case (ii), any natural numbers coprime to $p$ are units. So we similarly conclude that if a non-zero homogeneous element of $R$ is in the kernel of $d$, then its degree must be a multiple of $p$. That is
the kernel of the universal derivation
$d:R\rightarrow \Omega_{R/R_0}$
is contained in the Veronese subring $\bigoplus_{j\in \mathbb N} R_{pj}.$
This completes the proof. 
\end{proof}

\begin{remark}\label{mixed}
Our proof gives the following (possibly) mixed-characteristic version of 
Proposition \ref{kernel of the universal derivation}: {\it 
{Let $R$ be an $\mathbb{N}$-graded ring torsion free over $\mathbb Z$.
Then the kernel of the universal derivation $d: R\rightarrow \Omega_{R/R_0}$ is $R_0$.}} 
In particular, we can also deduce the following version of Theorem \ref{positive answer}: {\it {Let $R$ be an $\mathbb N$-graded ring without $\mathbb Z$-torsion. If $R$ is formally unramified over any subring contained in $R_0$, then $R=R_0$, so $R$ is reduced if $R_0$ is reduced.}}
\end{remark}

\section{Some Related Questions}
\noindent We end this note by mentioning a related question. The module of Kahler differentials $\Omega_{A/k}$  is the zeroth cohomology of the cotangent complex $\mathbb{L}_{A/k}.$ So a stronger condition than assuming that $\Omega_{A/k}=0$ would be that the entire 
complex $\mathbb{L}_{A/k}$ is exact. The following question appears as Question C.3, (ii), in \cite{AWS}, where it attributed to Bhargav Bhatt (also see \cite{Bhatt}).
\\\\
\textbf{Question 2:} 
Let $A$ be a $\mathbb{Q}$-algebra such that the cotangent complex $\mathbb{L}_{A/\mathbb{Q}}$ is quasi-isomorphic to the zero complex. Is $A$ reduced?

\end{document}